\documentclass[12pt]{article}
\usepackage{latexsym,amsfonts,amsmath,theorem,amssymb}
\usepackage{graphicx}

\usepackage{enumerate}
\usepackage{hyperref}

\newtheorem{theorem}{Theorem}

\newtheorem{proposition}[theorem]{Proposition}

\newenvironment{proof}{\begin{trivlist}
    \item[\hskip\labelsep{\bf Proof.}]}{$\hfill\Box$\end{trivlist}}

{\theoremstyle{plain} \theorembodyfont{\rmfamily}
\newtheorem{remark}[theorem]{Remark}}
{\theoremstyle{plain} \theorembodyfont{\rmfamily}
}

\allowdisplaybreaks

\newcommand{\bsgamma}{{\boldsymbol{\gamma}}}

\newcommand{\bst}{{\boldsymbol{t}}}

\newcommand{\bsx}{{\boldsymbol{x}}}

\newcommand{\bspitch}{{\boldsymbol{\,\pitchfork}}}
\newcommand{\bszero}{{\boldsymbol{0}}}
\newcommand{\bsone}{{\boldsymbol{1}}}
\newcommand{\rd}{\,\mathrm{d}}

\newcommand{\mask}[1]{}

\newcommand{\setu}{{\mathfrak{u}}}
\newcommand{\setv}{{\mathfrak{v}}}
\newcommand{\setw}{{\mathfrak{w}}}
\newcommand{\setU}{{\mathfrak{U}}}

\newcommand{\norm}[1]{\left\Vert#1\right\Vert}
\newcommand{\abs}[1]{\left\vert#1\right\vert}

\newcommand{\diam}{{\rm diam}}

\title{On Equivalence of Anchored and ANOVA Spaces;\\
      Lower Bounds}
\author{Peter Kritzer\thanks{P. Kritzer is supported by the Austrian
Science Fund (FWF):
Project F5506-N26, which is a part of the Special Research Program
"Quasi-Monte Carlo Methods:
Theory and Applications".} ,
Friedrich Pillichshammer\thanks{F. Pillichshammer is
partially supported by the Austrian Science Fund (FWF): Project F5509-N26,
which is a part of the Special Research Program "Quasi-Monte Carlo Methods:
Theory and Applications".}\;, and G. W. Wasilkowski}

\date{\today}

\begin{document}
\maketitle
\centerline{\it Dedicated to the memory of Joseph F. Traub (1932-2015)}

\begin{abstract}
We provide lower bounds for the norms of embeddings between
$\bsgamma$-weighted Anchored and ANOVA spaces of $s$-variate functions
with mixed partial derivatives of order one bounded in $L_p$ norm
($p\in[1,\infty]$). In particular we show that the norms
behave polynomially in $s$ for {\em Finite Order Weights} and
{\em Finite Diameter Weights} if $p>1$, and
increase faster than any polynomial in $s$ for
{\em Product Order-Dependent Weights} and any $p$.
\end{abstract}

\centerline{\begin{minipage}[hc]{130mm}{
}
\end{minipage}}

\section{Introduction}
In this short note, we continue research on the equivalence of
$\bsgamma$-weighted Anchored and ANOVA spaces of $s$-variate
functions with mixed partial derivatives of order one bounded in
$L_p$ norm ($p\in[1,\infty]$). The ANOVA spaces have been investigated
in a number of papers and many of their interesting properties
have been found. This includes small truncation and superposition
dimensions, see, e.g., \cite{Owen14} and papers cited there.
However, in general, these results cannot be utilized in practice since
ANOVA decomposition involves variances that are
impossible to compute numerically. On the other hand, small truncation or
superposition dimension for Anchored spaces is easy to exploit,
see, e.g., \cite{KrPiWa15,KSWW10b,Was14}. This is why it is
important to know for which weights $\bsgamma$ the Anchored and
ANOVA spaces are equivalent, or more precisely how fast the norms
of the corresponding embeddings increase when the number $s$ of
variables increases. If the norms are uniformly bounded then we say
that {\em the spaces are uniformly equivalent}. If the norms increase like
a polynomial in $s$ then we say that {\em the spaces are polynomially
equivalent.}

This question of equivalence was first addressed in \cite{HeRi13} for product
weights and in the Hilbert space setting ($p=2$). More precisely the
authors provided lower and upper bounds on the norms of the
corresponding embeddings and concluded that there
is uniform equivalence if and only if the
weights are summable. This result was later slightly strengthened in
\cite{KrPiWa15} by showing that the embedding and its inverse have
the same norm and by delivering an exact formula for it.

In \cite{HeRiWa15}, exact values of the embedding norm were delivered
for general weights but only with $p=1$ and $p=\infty$. The results
were then applied to various
types of weights to see when there is uniform
or polynomial 
equivalence. In particular, it was
shown that for an important class of {\em Product Order-Dependent}
weights there is no polynomial equivalence.

Next in \cite{SH}, the authors applied the complex interpolation
in conjunction with the results mentioned above to get interesting
upper bounds on the norms of the embeddings for general weights and
for all values of $p\in(1,\infty)$.

With the exception of product weights, 
the result of \cite{SH} does not provide general lower bounds for the
norms of the embeddings. This is why, in this short note, we
prove that the norms of the embedding and its inverse are equal and
provide lower bounds for them. These lower bounds are sharp for
$p=1$ and $p=\infty$ and general weights. 
They are also sharp
for {\em finite order weights} and
{\em finite diameter weights} 
for any $p$, and we believe that they
are sharp for any $p$ and general weights.

We next use these lower bounds for the following three cases of weights.

{\bf Finite Order Weights:} It was shown in \cite{HeRiWa15} that
the norms are uniformly bounded for $p=1$ and are proportional
to a polynomial in $s$ for $p=\infty$. Then it was concluded
in \cite{SH} that the embedding norms are bounded from above by a
polynomial in $s$ for every $p\in(1,\infty)$. Using our lower bounds
we show that these norms are indeed polynomial in $s$ for $p>1$.
More precisely they are equal to $\Theta(s^{q/p^*})$ where $p^*$
is the conjugate of $p$ and $q$ is the order of the weights.

{\bf Finite Diameter Weights:} Such kinds of weights have not been
investigated in
this context so far. We use our lower bound to show that
for $p\in(1,\infty)$,
the norms of the embeddings increase at least as fast as 
$s^{1/p^*}$, where $p^*$ is the conjugate of $p$.
We also deliver matching upper bounds to conclude that the norms
are equal to $\Theta(s^{1/p^*})$.

{\bf Product Order-Dependent Weights:} Since there is no
polynomial equivalence for $p=1$ and $p=\infty$, the techniques
employed in \cite{SH} could not answer whether there is polynomial
equivalence for $p\in(1,\infty)$ or not. Using our lower bounds we conclude
that for $p\in(1,\infty)$, the norms of the embeddings increase
faster than any polynomial in $s$.

\section{Basic Facts and Notation}
Following \cite{HeRiWa15}, 
we recall basic facts and notation
pertaining to the {\em Anchored} and {\em ANOVA} spaces considered in
this paper.

We begin with the notation that is used in this paper: by
$s \in \mathbb{N}$ we
denote the dimension, and the set of coordinate indices is
\[
  [s]\,=\,\{1,2,\dots,s\}.
\]
By $\setu,\setv,\setw$ we denote subsets of $[s]$. For
$\setu \subseteq [s]$ its complement is denoted by
$\setu^c=[s]\setminus \setu$. Moreover, for
$\bst,\bsx\in\mathbb{R}^s$, where $\bst=(t_1,\ldots,t_s)$
and $\bsx=(x_1,\ldots,x_s)$, and $\setu \subseteq[s]$, we define
\[
  [\bsx_\setu;\bst_{\setu^c}]\,=\,(y_1,\dots,y_s)\quad\mbox{with}\quad
  y_j=\left\{\begin{array}{ll} x_j &\mbox{if\ }j\in\setu,\\
  t_j &\mbox{if\ }j\notin\setu.
  \end{array}\right.
\]
We also write $\bsx_\setu$ to denote the $|\setu|$-dimensional
vector $(x_j)_{j\in\setu}$ and
\[
   f^{(\setu)}\,=\,\frac{\partial^{|\setu|}f}{\partial \bsx_\setu}
   \,=\,\prod_{j\in\setu}\frac{\partial}{\partial x_j} f\quad
   \mbox{with}\quad f^{(\emptyset)}\,=\,f.
\]

\subsection{Anchored Spaces}
For $p\in[1,\infty]$, let $F_p=W^1_{p,0}$ be the space of functions
defined on $D=[0,1]$ that are absolutely continuous, vanish at zero,
and with the derivative bounded in the $L_p$ norm. It is a Banach
space with respect to the norm $\|f\|_{F_p}=\|f'\|_{L_p}$. The space $F_p$
is the building block for the anchored spaces of $s$-variate functions.

For non-empty $\setu$, let $F_{\setu,p}$ be a Banach space that is
the completion of the space spanned by
$f(\bsx)=\prod_{j\in\setu}f_j(x_j)$ with $f_j\in F_p$ with respect to
the norm
\[
   \|f\|_{F_{\setu,p}}\,=\,\|f^{(\setu)}\|_{L_p}.
\]
For $\setu=\emptyset$, $F_{\emptyset,p}$ is the space of constant
functions with the norm given by the absolute value.

Consider next a family $\bsgamma=(\gamma_\setu)_{\setu\subseteq[s]}$
of non-negative numbers, called {\em weights}. Let
\[
   \setU\,=\,\{\setu \subseteq [s]\ :\ \gamma_\setu\,>\,0\}.
\]
The corresponding $\bsgamma$-weighted {\em anchored space}
$F_{s,p,\bsgamma}$ is the completion of
$\bigoplus_{\setu\in\setU}F_{\setu,p}$
with respect to the norm given by
\[
   \|f\|_{F_{s,p,\bsgamma}}\,=\,\left(\sum_{\setu\in\setU}\gamma_\setu^{-p}
   \,\|f^{(\setu)}([\cdot_\setu;\bszero_{\setu^c}])\|_{L_p}^p
    \right)^{1/p}.
\]

It is known, see, e.g., \cite[Section~3]{HeRiWa15}, that for non-empty $\setu$,
\[
F_{\setu,p}\,=\,T_{\setu,p}(L_p(D^\setu)),
\]
where
\[
  T_{\setu,p}(h)(\bsx)\,=\,\int_{D^\setu}h(\bst_\setu)\,
   \prod_{j\in\setu} \bsone_{[0,x_j)}(t_j) \rd\bst_\setu
\]
and $\bsone_{I}$ is the characteristic function of the set $I$,
i.e., $\bsone_I(t)=1$ if $t \in I$ and 0 otherwise.
Then any $f\in F_{s,p,\bsgamma}$ has a unique decomposition,
called {\em anchored decomposition},
\[
   f\,=\,\sum_{\setu\in\setU}f_{\setu}\quad\mbox{with}\quad
   f_\setu=T_{\setu,p}(h_\setu)\quad\mbox{for}\quad h_\setu\in
   L_p(D^\setu),
\]
where here and in the following we write $D^{\setu}$
for $D^{|\setu|}$. Moreover, $f_\setu\in F_{\setu,p}$ and
\[
   f^{(\setu)}_\setu\,=\,h_\setu\,=\,f^{(\setu)}([\cdot_\setu;\bszero_{\setu^c}]).
\]

\subsection{ANOVA Spaces}
The definition of the corresponding $\bsgamma$-weighted ANOVA spaces is very
similar to that of Anchored spaces with the only difference that
instead of the space
$F_p=W^1_{p,0}$ we use the space $W^1_{p,{\rm int}}$ of absolutely
continuous functions on $[0,1]$ with $\|f'\|_{L_p}<\infty$ and such
that
\[
  \int_0^1f(t)\rd t\,=\,0.
\]
Then the corresponding $\bsgamma$-weighted {\em ANOVA space}
$H_{s,p,\bsgamma}$ has the norm given by
\[
   \|f\|_{H_{s,p,\bsgamma}}\,=\,\left(\sum_{\setu\in\setU}\gamma_\setu^{-p}\,
   \left\|\int_{D^{\setu^c}}f^{(\setu)}([\cdot_\setu;\bst_{\setu^c}])
   \rd\bst_{-\setu}\right\|_{L_p}^p\right)^{1/p}.
\]
Any function from $H_{s,p,\bsgamma}$ has a unique {\em ANOVA decomposition}
\[
  f\,=\,\sum_{\setu\in\setU}f_\setu,
\]
where
\[
  f_\setu(\bsx)\,=\,\int_{D^\setu}g_\setu(\bst_\setu)\,
  \prod_{j\in\setu}(\bsone_{[0,x_j)}(t_j)-(1-t_j))\,\rd\bst_\setu
\]
for some $g_\setu\in L_p(D^\setu)$ and
\[
  f_\setu^{(\setu)}\,=\,g_\setu\,=\,\int_{D^{\setu^c}}f^{(\setu)}([\cdot_\setu;
  \bst_{\setu^c}])\rd\bst_{\setu^c}.
\]

\section{Equivalence of Anchored and ANOVA Spaces}
It was shown in \cite{HeRiWa15} that the Anchored and ANOVA spaces
are equal (as sets of functions) if and only if the following holds:
\begin{equation}\label{ass-1}
  \gamma_\setw\,>0\,\quad\mbox{implies that}\quad\gamma_\setu\,>\,0
  \mbox{\ for all\ }\setu\,\subset\,\setw.
\end{equation}
This is why from now on we assume that \eqref{ass-1} is satisfied.

Let
\[
  \imath\, = \,\imath_{s,p,\bsgamma}:F_{s,p,\bsgamma} \, \hookrightarrow \,
  H_{s,p,\bsgamma}
\]
be the {\it embedding operator}, $\imath(f)=f$.
As mentioned in the introduction, \cite{SH} provides interesting
upper bounds on the norms of $\imath$ and its inverse $\imath^{-1}$.

We now prove that $\imath$ and $\imath^{-1}$ have the
same norm. Moreover we provide a lower bound for that norm which,
as will be illustrated, is sharp for a number of special cases.

\begin{theorem}
The norms $\|\imath\|_{F_{s,p,\bsgamma}\hookrightarrow H_{s,p,\bsgamma}}$
and $\|\imath^{-1}\|_{H_{s,p,\bsgamma}\hookrightarrow F_{s,p,\bsgamma}}$
are the same and are equal to
\begin{equation}\label{equal}
  \sup_{\{c_\setu,h_\setu\}_{\setu\in\setU}}\frac{\left(
   \sum_{\setv\in\setU}\gamma_\setv^{-p}\int_{D^\setv}\abs{
  \sum_{\setw\subseteq [s] \setminus\setv}c_{\setv\cup\setw}\int_{D^\setw}
  h_{\setv\cup\setw}([\bsx_\setv;\bst_\setw])\,\prod_{j\in\setw}(1-t_j)
  \rd\bst_\setw}^p \rd\bsx_\setv\right)^{1/p}}
  {\left(\sum_{\setu\in \setU}c_\setu^p\right)^{1/p}},
\end{equation}
where the supremum is with respect to non-negative numbers $c_\setu$
and functions $h_\setu\in L_{p}(D^\setu)$
such that $\|h_\setu\|_{L_{p}(D^\setu)}=\gamma_\setu$.

The norms are bounded from below by
\begin{equation}\label{lower}
  \sup_{\{c_\setu\ge0\}_{\setu\in\setU}}\frac{\left(\sum_{\setv\in\setU}
   \gamma_\setv^{-p}\left(\sum_{\setw\subseteq[s]\setminus\setv}
   c_{\setv\cup\setw}\,\gamma_{\setv\cup\setw} m_{p^*}^{\abs{\setw}}
   \right)^p\right)^{1/p}}
   {\left(\sum_{\setu\in\setU} c_\setu^p\right)^{1/p}},
\end{equation}
where
\[
m_{p^*}\,=\,\frac1{(p^*+1)^{1/p^*}} \ \ \ \
\mbox{ and }\ \ \frac{1}{p}+\frac{1}{p^*}=1
\]
with $m_{\infty}=1$.  
\end{theorem}

\begin{proof}
As already mentioned, any $f\in F_{s,p,\bsgamma}$ can be written as
\[
  f(\bsx)\,=\,c_\emptyset\,\gamma_\emptyset +
  \sum_{\emptyset\not=\setu\in\setU} f_{\bspitch,\setu}(\bsx),\quad\mbox{where}\quad
  f_{\bspitch,\setu}(\bsx)\,=\, c_\setu\, \int_{D^{\setu}}  h_\setu(\bst_\setu)\,
  \prod_{j\in\setu} \bsone_{[0,x_j)}(t_j)\rd\bst_\setu,
\]
for non-negative numbers $c_\setu$ and functions $h_\setu$ such that
\[
  \|h_\setu\|_{L_{p}(D^\setu)}\,=\,\gamma_\setu.
\]
Of course, the terms $f_{\bspitch,\setu}$ are from the anchored
decomposition of $f$, and
\[
  f_{\bspitch,\setu}^{(\setu)}\,=\,c_\setu h_\setu.
\]
Therefore,
\[
\|f\|_{F_{s,p,\bsgamma}}\,=\,\left(\sum_{\setu\in\setU}c_\setu^p\right)^{1/p}.
\]
For $\setv\subseteq[s]$, $f_{\bspitch,\setu}^{(\setv)}=0$ if
$\setv\not\subseteq\setu$.
Consider therefore $\setu$ that contains $\setv$.  Then
\[
\int_{D^{\setv^c}} f_\setu^{(\setv)}([\bsx_\setv;\bsx_{\setv^c}])
\rd\bsx_{\setv^c} \,=\,c_\setu\int_{D^{\setu\setminus\setv}}
  h_\setu(\bsx_\setv;\bst_{\setu\setminus\setv})\prod_{j\in\setu\setminus\setv}
  (1-t_j)\rd\bst_{\setu\setminus\setv}
\]
and the ANOVA term $f_{A,\setv}$ is given by
\[
f_{A,\setv}(\bsx)\,=\,\sum_{\setw\subseteq[s]\setminus\setv}
c_{\setv\cup\setw}
  \int_{D^\setw}h_{\setv\cup\setw}([\bsx_\setv;\bst_\setw])\,\prod_{j\in\setw}
  (1-t_j)\rd \bst_\setw.
\]
Therefore
\begin{eqnarray*}
  \|f\|_{H_{s,p,\bsgamma}}&=&\left(  \gamma_\emptyset^{-p}
\left| c_\emptyset\,\gamma_\emptyset+
  \sum_{\emptyset\not=\setu\in\setU}c_\setu \int_{D^\setu}h_\setu(\bst_\setu)
  \,\prod_{j\in\setu}(1-t_j)\rd\bst_\setu\right |^p \right.\\
  & & \left.
  \quad+\sum_{\emptyset\not=\setv\in\setU}\gamma_\setv^{-p}\,\int_{D^\setv}
 \left |\sum_{\setw\subseteq[s]\setminus\setv}c_{\setv\cup\setw}
  \int_{D^\setw}h_{\setv\cup\setw}([\bsx_\setv;\bst_\setw])\,\prod_{j\in\setw}
  (1-t_j)\rd \bst_\setw\right|^p \rd\bsx_\setv\right)^{1/p}
\end{eqnarray*}
and the norm of the embedding $\imath$ is given by the supremum
of the right hand side of the above equation divided by
$\left(\sum_{\setu\subseteq[s]}c_\setu^p\right)^{1/p}$. The supremum
is with respect to non-negative numbers $c_\setu$ and functions
$h_\setu\in L_p (D^\setu)$ such that
$\|h_\setu\|_{L_p(D^\setu)}=\gamma_\setu$.

Consider now $f\in H_{s,p,\bsgamma}$. It can be written as
\[
  f(\bsx)\,=\,c_\emptyset\,\gamma_\emptyset+
   \sum_{\emptyset\not=\setu\in\setU}f_{A,\setu}(\bsx),
\]
where
\[
 f_{A,\setu}(\bsx)\,=\,(-1)^{\abs{\setu}}
  c_\setu\,\int_{D^{\setu}}h_\setu(\bst_\setu)\,\prod_{j\in\setu}
  \left(\bsone_{[0,x_j)}(t_j)-(1-t_j)\right)\rd\bst_\setu.
\]
The terms $f_{A,\setu}$ are from the ANOVA decomposition of $f$.
Again, we choose $\|h_\setu\|_{L_p}=\gamma_\setu$ and, therefore,
\[
\|f\|_{H_{s,p,\bsgamma}}\,=\,\left(\sum_{\setu\in \setU}
c_\setu^p\right)^{1/p}.
\]
For $\setv\subseteq\setu$,
\[
  f_{A,\setu}^{(\setv)}([\bsx_\setv;\bszero_{\setv^c}])\,=\,
  (-1)^{\abs{\setv}}c_\setu \int_{D^{\setu\setminus\setv}}h_\setu
    ([\bsx_\setv;\bst_{\setu\setminus\setv}])\prod_{j\in\setu\setminus\setv}
  (1-t_j)\rd\bst_{\setu\setminus\setv}.
\]
Therefore the anchored norm of $f$ is given by
\begin{eqnarray*}
  \|f\|_{F_{s,p,\bsgamma}}&=&\left(  \gamma_\emptyset^{-p}
\left|c_\emptyset\,\gamma_\emptyset+
  \sum_{\emptyset\not=\setu\in\setU}c_\setu \int_{D^\setu}h_\setu(\bst_\setu)
  \,\prod_{j\in\setu}(1-t_j)\rd\bst_\setu\right|^p \right.\\
  & & \left.
  \quad+\sum_{\emptyset\not=\setv\in\setU}\gamma_\setv^{-p}\,\int_{D^\setv}
 \left|\sum_{\setw\subseteq[s]\setminus\setv}c_{\setv\cup\setw}
  \int_{D^\setw}h_{\setv\cup\setw}([\bsx_\setv;\bst_w])\,\prod_{j\in\setw}
  (1-t_j)\rd \bst_\setw\right|^p  
\rd\bsx_\setv\right)^{1/p}.
\end{eqnarray*}
This proves \eqref{equal}.

We now prove the lower bound \eqref{lower}. Consider
\[
  h_\setu(\bsx)\,=\,\gamma_\setu\,\prod_{j\in\setu}h(x_j),
\]
where the univariate function $h\in L_{p}(D)$ is such that
\[
   \|h\|_{L_{p}(D)}\,=\,1\quad\mbox{and}\quad
   \int_D h(t)\,(1-t)\rd t\,=
   \,\|h\|_{L_{p}(D)}\norm{(1-\cdot)}_{L_{p^*}(D)}\,
   =m_{p^*}.
\]
Then
\begin{eqnarray*}
\lefteqn{\int_{D^\setv}\left(\sum_{\setw\subseteq[s]\setminus\setv}
 c_{\setv\cup\setw}\int_{D^\setw}
 h_{\setv\cup\setw}([\bsx_\setv;\bst_\setw])
  \,\prod_{j\in\setw}(1-t_j)\rd\bst_\setw\right)^p
  \rd \bsx_\setv}\\
 &=&\int_{D^\setv}\prod_{j\in\setv}h^p(x_j)
   \rd\bsx_\setv\,
    \left(\sum_{\setw\subseteq[s]\setminus\setv} c_{\setv\cup\setw}\,
  \gamma_{\setv\cup\setw} m_{p^*}^{\abs{\setw}}\right)^p\\
  &=&
    \left(\sum_{\setw\subseteq[s]\setminus\setv} c_{\setv\cup\setw}\,
  \gamma_{\setv\cup\setw}m_{p^*}^{\abs{\setw}}\right)^p.
\end{eqnarray*}
This completes the proof.
\end{proof}

We have the following proposition.
\begin{proposition}
  The lower bound \eqref{lower} is sharp for $p=1$ and $p=\infty$
  and arbitrary weights. 
\end{proposition}
\begin{proof}
To simplify the notation let us use
\[
  \overline{\gamma}_\setu\,=\,\gamma_{\setu}\, m_{p^*}^{|\setu|}\,=\,
  \frac{\gamma_\setu}{(p^*+1)^{|\setu|/p^*}}.
\]
Then the numerator in the lower bound \eqref{lower}
can be rewritten as
\[
 N_p((c_\setu)_\setu)\,=\,\left(\sum_{\setv\in\setU}
\overline{\gamma}_\setv^{-p}\,\left(\sum_{\substack{\setu\subseteq[s]\\
\setv\subseteq\setu}}c_\setu\,
  \overline{\gamma}_{\setu}\right)^p\right)^{1/p}.
\]
For $p=1$ we have $p^*=\infty$. Hence $(p^*+1)^{1/p^*}=1$ and
\[
  \frac{N_p((c_\setu)_\setu)}{\sum_{\setu\in\setU}c_\setu}\,\le\,
  \max_{\setu\subseteq[s]}\sum_{\setv\subseteq\setu}\frac{{\gamma_\setu}}
   {{\gamma_\setv}}
\]
and the inequality above is sharp. From \cite{HeRiWa15} we know that
for $p=1$,
\[
  \|\imath\|_{F_{s,p,\bsgamma}\hookrightarrow H_{s,p,\bsgamma}}\,=\,
  \max_{\setu\in\setU}\sum_{\setv\subseteq\setu}\frac{{\gamma_\setu}}
   {{\gamma_\setv}}.
\]
This proves the claim for $p=1$.

Consider next $p=\infty$. Then, of course, $(p^*+1)^{1/p^*}=2$.
The lower bound \eqref{lower} can be rewritten as
\[
  \max_{(c_\setu\ge0)_{\setu\in\setU}}\frac{\max_{\setv\in\setU}\gamma_\setv^{-1}\,
  \sum_{\setw\subseteq[s]\setminus\setv}c_{\setv\cup\setw}\frac{\gamma_{\setv\cup\setw}}
  {2^{\setw}}}{\max_{\setu\subseteq[s]}c_\setu}\,=\,
  \max_{\setv\in\setU}\sum_{\setw\subseteq[s]\setminus\setv}
  \frac{\gamma_{\setv\cup\setw}}{2^{|\setw|}\,\gamma_{\setv}}.
\]
This lower bound for $p=\infty$ is also sharp since it is equal to the
norm of the corresponding embedding, as shown in \cite{HeRiWa15}.
\end{proof}

As already mentioned, \cite{SH} provides an upper bound on the
norm of the embedding for product weights and any $p\in(1,\infty)$.
It also provides a matching lower bound. The purpose of the
proposition below is to show that a sharp lower bound can also be
obtained from \eqref{lower}.
Recall that {\em product weights}, introduced in \cite{SlWo98}, 
 are of the form
\[
   \gamma_\setu\,=\,\prod_{j\in\setu}\gamma_j\quad
  \mbox{for a sequence $(\gamma_j)_{j \ge 1}$ in $\mathbb{R}^+$.}
\]
In particular, for product weights we have $\setU=\mathcal{P}([s])$,
the set of all subsets of $[s]$.

\begin{proposition}
For $p\in(1,\infty)$,
the lower bound \eqref{lower} is
(modulo a multiplicative constant) sharp for product
weights and
\begin{equation}\label{love-it-too}
  \|\imath\|_{F_{s,p,\bsgamma}\hookrightarrow H_{s,p,\bsgamma}}\,
  \ge\,\prod_{j=1}^s\left(1+\gamma_j\,
  \left(\frac{p-1}{p^*+1}\right)^{1/p^*}\right)^{1/p}.
\end{equation}
In particular it shows that for the uniform equivalence
it is necessary that
\[
  \sum_{j=1}^\infty \gamma_j\,<\,\infty.
\]
\end{proposition}
\begin{proof}
Consider $c_\setu=\prod_{j\in \setu}c_j$ with $c_j \ge 0$ for $j \ge 1$.
Then the numerator in the  expression in the lower bound \eqref{lower}
is equal to
\begin{eqnarray*}
\left(\sum_{\setv\subseteq[s]}c_\setv^p\,\left(
  \sum_{\setw\subseteq[s]\setminus\setv}c_\setw\,\gamma_\setw\,m_{p^*}^{|\setw|}
   \right)^p\right)^{1/p}
&=&\,\left(\sum_{\setv\subseteq[s]}c_\setv^p
  \left(\sum_{\setw\subseteq[s]\setminus\setv}\prod_{j\in\setw}(c_j\,\gamma_j
   \, m_{p^*})\right)^p\right)^{1/p}\\
&=&\,\left(\sum_{\setv\subseteq[s]}\prod_{j\in\setv}c_j^p
  \,\prod_{j\in[s]\setminus\setv}(1+c_j \gamma_j m_{p^*})^p\right)^{1/p}\\
&=&\,\left(\prod_{j=1}^s\left(c_j^p+\left(1+c_j\,\gamma_j\,
      m_{p^*}\right)^p\right)\right)^{1/p}\\
&=&\,\prod_{j=1}^s\left(1+c_j^p+\left(1+c_j\,\gamma_j
   \,m_{p^*}\right)^p-1\right)^{1/p}.
\end{eqnarray*}
Since the denominator is equal to $\prod_{j=1}^s(1+c_j^p)$,
we get that the norm of the embedding is bounded from below by
\begin{equation}\label{lbdprodweight}
  \sup_{\{c_j\ge0\}_{j=1}^s} \prod_{j=1}^s\left(1+\frac{
   (1+c_j\,\gamma_j\,m_{p^*})^p-1}{1+c_j^p}\right)^{1/p}.
\end{equation}
Clearly, for $p\in(1,\infty)$ and $c_j=1/(p-1)^{1/p}$
\begin{eqnarray*}
   \left(1+\frac{(1+c_j\,\gamma_j\,m_{p^*})^p-1}{1+c_j^p}\right)^{1/p}
   &\ge&\left(1+\frac{p\,c_j\,\gamma_j\,m_{p^*}}{1+c_j^p}\right)^{1/p}\\
  &=&\left(1+\gamma_j\,\left(\frac{p-1}{p^*+1}\right)^{1/p^*}\right)^{1/p}\\
  &=&1+\gamma_j\,p^{-1}\,\left(\frac{p-1}{p^*+1}\right)^{1/p^*}+
   O\left(\gamma_j^2\right).
\end{eqnarray*}
This completes the proof.
\end{proof}

\begin{remark}
For $p=1$ we have $m_{p^*}=m_{\infty}=1$ and hence we get from
\eqref{lbdprodweight}
\[
\sup_{\{c_j\ge0\}_{j=1}^s} \prod_{j=1}^s\left(1+\frac{
  (1+c_j\,\gamma_j\,m_{p^*})^p-1}{1+c_j^p}\right)^{1/p}\,=\,
\sup_{\{c_j\ge0\}_{j=1}^s} \prod_{j=1}^s\left(1+\frac{
  c_j}{1+c_j}\,\gamma_j\right)\,=\,\prod_{j=1}^s (1+\gamma_j).
\]
This matches exactly the result from \cite[Proposition~17 for $p=1$]{HeRiWa15}.

For $p=2$ we know from \cite{KrPiWa15} that
\begin{equation}\label{true}
  \|\imath\|_{F_{s,2,\bsgamma}\hookrightarrow H_{s,p,\bsgamma}}\,=\,
   \prod_{j=1}^s\left(1+\frac{\gamma_j}{\sqrt{3}}\,
   \left(\sqrt{1+\frac{\gamma_j^2}{12}}+\frac{\gamma_j^3}{\sqrt{12}}\right)
    \right)^{1/2}.
\end{equation}
Note that the lower bound \eqref{love-it-too} for $p=2$ takes the form
\[
  \prod_{j=1}^s\left(1+\frac{\gamma_j}{\sqrt{3}}\right)^{1/2}
\]
and is very close to the true value in \eqref{true}.
\end{remark}

The following proposition provides a lower bound that is
sometimes easier to use.

\begin{proposition}\label{prop:ll}
The lower bound \eqref{lower} is bounded from below by
\begin{equation}\label{lower-lower}
  \max_{\setu\in\setU}\left(\sum_{\setv\subseteq\setu}\frac{\gamma_\setu^p}
{(p^*+1)^{p|\setu\setminus\setv|/p^*}\,\gamma_\setv^p}\right)^{1/p}.
\end{equation}
\end{proposition}

\begin{proof}
Clearly the numerator in \eqref{lower} is not smaller than
\[
\left(\sum_{\setv\in\setU} \gamma_\setv^{-p}\sum_{\setw\subseteq[s]\setminus\setv}
c_{\setv\cup\setw}^p\,\frac{\gamma_{\setv\cup\setw}^p}{(p^*+1)^{p|\setw|/p^*}}
\right)^{1/p}\,=\,\left(\sum_{\setu\in\setU}c_\setu^p\,\gamma_\setu^p
\sum_{\setv\subseteq\setu}\gamma_\setv^{-p}\,(p^*+1)^{-p|\setu\setminus\setv|/p^*}
\right)^{1/p}.
\]
Let $\setu^*$ be such that
\[
\max_{\setu\in\setU}\sum_{\setv\subseteq\setu}\frac{\gamma^p_\setu}
{(p^*+1)^{p|\setu\setminus\setv|/p^*}\,\gamma_\setv^p}
\]
is attained at $\setu^*$. Then taking $c_{\setu^*}=1$ and $c_\setu=0$ for
all $\setu\not=\setu^*$ completes the proof.
\end{proof}

We now apply \eqref{lower} to get lower bounds for special classes of
weights.

\subsection{Finite Order Weights}
Consider {\em finite order weights}, for the first time dealt with in 
\cite{DSWW06}, of the form
\begin{equation}\label{def-FOW}
  \gamma_\setu\,=\,\left\{\begin{array}{ll} \omega^{|\setu|} &
    \mbox{if\ }|\setu|\le q,\\
  0 &\mbox{if\ }|\setu|>q.\end{array}\right.
\end{equation}
It was shown in \cite{HeRiWa15} that the norm of the embedding is
uniformly bounded for $p=1$, and grows like a polynomial in $s$
for $p=\infty$. The authors of \cite{SH} proved using the above
results and complex interpolation theory that for $p\in(1,\infty)$
the norm of the embedding is bounded from above by a polynomial in
$s$. It is therefore of interest to see whether the embedding norm
is uniformly bounded or indeed behaves polynomially for $p\in(1,\infty)$.

\begin{proposition}\label{FOW}
For finite order weights \eqref{def-FOW} and $p\in(1,\infty)$ we have
\[
  \|\imath\|_{F_{s,p,\bsgamma}\hookrightarrow  H_{s,p,\bsgamma}}\,=\,
  \Theta\left(s^{q/p^*}\right).
\]
\end{proposition}
\begin{proof}
It was shown in \cite{HeRiWa15} that the norm of the embedding
for $p=1$ is uniformly bounded and it is proportional to
$s^q$ for $p=\infty$. Hence the results of \cite{SH} imply the
following upper bound
\[
\|\imath\|_{F_{s,p,\bsgamma}\hookrightarrow  H_{s,p,\bsgamma}}\,=\,
  O\left(s^{q\,(1-1/p)}\right).
\]
Hence it is enough to show a matching lower bound. For that
purpose we will use \eqref{lower} for a special choice of $c_\setu$. Namely,
consider $c_\setu=1$ if $|\setu|=q$ and $c_\setu=0$ otherwise.
Note that then the denominator in \eqref{lower} equals
\[
 \left(\begin{array}{c} s \\ q \end{array}\right)^{1/p}.
\]
Consider next the numerator of \eqref{lower} where instead of the
whole summation with respect to $\setv$ we consider only one term
with $\setv=\emptyset$. Then we have
\[
  \sum_{|\setw|=q} c_\setw\,\gamma_\setw\,
    m_{p^*}^{|\setw|}
  \,=\,m_{p^*}^q \omega^q\,\left(\begin{array}{c} s \\ q \end{array}\right).
\]
Therefore
\[
   \|\imath\|_{F_{s,p,\bsgamma}\hookrightarrow
     H_{s,p,\bsgamma}}\,\ge\,m_{p^*}^q \omega^q\,
   \left(\begin{array}{c} s \\ q \end{array}\right)^{1-1/p}
   \,\ge\, \frac{m_{p^*}^q\omega^q}{(q!)^{1-1/p}}\,(s-q)^{q\,(1-1/p)}.
\]
This completes the proof.
\end{proof}

\subsection{Finite Diameter Weights}
Consider {\em finite diameter weights}, which were first introduced by Creutzig 
(see \cite{C07}, and also \cite{NW08})
of the form
\begin{equation}\label{def-FDW}
  \gamma_\setu\,=\,\left\{\begin{array}{ll} \omega^{|\setu|} &
    \mbox{if\ } \diam(\setu)\le q,\\
  0 &\mbox{if\ } \diam(\setu)>q,\end{array}\right.
\end{equation}
where $\diam(\setu)=\max_{i,j \in \setu}|i-j|$ and
$\diam(\emptyset)=0$ by convention.

\begin{proposition}\label{FDW} 
For finite diameter weights \eqref{def-FDW} and $p\in[1,\infty]$ we have
\begin{equation}\label{love-it}
  \|\imath\|_{F_{s,p,\bsgamma}\hookrightarrow  H_{s,p,\bsgamma}}\,=\,
  \Theta\left((s-q)^{1/p^*}\right).
\end{equation}
\end{proposition}
\begin{proof}
We will use the following fact in a number of places. For 
$\ell\in\{2,\ldots,q\}$,
\begin{equation}\label{eqell}
  \sum_{\substack{\setu \subseteq [s]\\ \diam(\setu)=\ell}}x^{|\setu|}  = 
  \sum_{k=2}^{\ell+1} x^k  \sum_{\substack{\setu \subseteq [s],\
    |\setu|=k \\ \diam(\setu)=\ell}} 1 = \sum_{k=2}^{\ell+1} x^k (s-\ell)
      {\ell-1 \choose k-2}=(s-\ell) x^2 (1+x)^{\ell-1}.
\end{equation}

We start by showing the result for the case $p=1$. In this case, the 
embedding norm is equal to
\[
 \max_{\setu\in\setU}\sum_{\setv\subseteq \setu}\frac{\gamma_\setu}{\gamma_\setv}
 =\max_{\substack{\setu\subseteq [s]\\ \diam (\setu)\le q}}
 \omega^{\abs{\setu}}\sum_{\setv\subseteq \setu}\omega^{-\abs{\setv}}
 =\max_{\substack{\setu\subseteq [s]\\ \diam (\setu)\le q}}
 (1+\omega)^{\abs{\setu}}=(1+\omega)^{q+1},
\]
which yields the desired result for $p=1$. 
For $p=\infty$, the embedding norm is equal to
\begin{eqnarray*}
\max_{\setv\in\setU}\sum_{\substack{\setw\subseteq[s]\setminus\setv
    \\ \diam(\setw\cup\setv)\le q}}\frac{\omega^{|\setw|}}{2^{|\setw|}}
&=&\sum_{\substack{\setw\in\setU\\ \diam (\setw)\le q}}
\left(\frac{\omega}{2}\right)^{|\setw|}\\
&=&\sum_{\ell=0}^q \left(\frac{\omega}{2}\right)^{\ell}  
\sum_{\substack{\setw\in\setU \\ \diam(\setw)=\ell}}1\\
&=&\sum_{\substack{\setw\in\setU \\ \diam(\setw)=0}}1+
\sum_{\substack{\setw\in\setU \\ \diam(\setw)=1}}1+
\sum_{\ell=2}^q \left(\frac{\omega}{2}\right)^{\ell}  
\sum_{\substack{\setw\in\setU \\ \diam(\setw)=\ell}}1\\
&=&s+1 +\frac{\omega}{2}(s-1) + 
\sum_{\ell=2}^q \left(\frac{\omega}{2}\right)^{\ell}
(s-q)2^{q-1}\\
&=&\Theta\left(s-q\right),
\end{eqnarray*}
where we used \eqref{eqell} in the fourth equality. 
Therefore, using the result of \cite{SH} we get that for $p\in(1,\infty)$
the corresponding norm is bounded by
\[
O\left((s-q)^{1-1/p}\right).
\]

Hence to complete the proof, we need to show a matching lower bound.
Similar to the case of finite order weights we
consider $c_\setu=1$ if $\diam(\setu)=q$ and $c_\setu=0$ otherwise.
Then, using \eqref{eqell} again, the denominator in \eqref{lower} equals
\[
 \left(\sum_{\diam(\setu)=q} 1\right)^{1/p}\, =\, (s-q)^{1/p} 2^{(q-1)/p}.
\]
Consider next the numerator of \eqref{lower} where instead of the
whole summation with respect to $\setv$ we consider only one term
with $\setv=\emptyset$. Then we have, once more by \eqref{eqell},
\[
  \sum_{\diam(\setw)=q} c_\setw\,\gamma_\setw\,m_{p^*}^{|\setw|}
  \,=\, \omega^q \sum_{\diam(\setw)=q} m_{p^*}^{|\setw|}\, =\,
 \omega^q (s-q) m_{p^*}^2 (1+m_{p^*})^{q-1}.
\]
Therefore
\[
\|\imath\|_{F_{s,p,\bsgamma}\hookrightarrow H_{s,p,\bsgamma}}\,\ge\,
\frac{\omega^q (s-q) m_{p^*}^2 (1+m_{p^*})^{q-1}}{(s-q)^{1/p}\, 2^{(q-1)/p}}
\,=\,\omega^q m_{p^*}^2  \left(\frac{1+m_{p^*}}{ 2^{1/p}}\right)^{q-1} (s-q)^{1/p^*}.
\]
This completes the proof.
\end{proof}

\subsection{Product Order-Dependent Weights}
Consider {\em product order-dependent weights}, as introduced in \cite{KSS12}, of the form
\[
  \gamma_\setu\,=\,(|\setu|!)^{\beta_1}\,\prod_{j\in\setu}\frac{c}{j^{\beta_2}}
\quad\mbox{for\ }c>0\mbox{\ and\ }0\,<\,\beta_1\,<\,\beta_2.
\]
It was shown in \cite{HeRiWa15} that for $p=1$ or $p=\infty$,
the norm of the embedding
converges to infinity faster than any polynomial in $s$.

\begin{proposition}\label{POD}
For product order-dependent weights and $p\in[1,\infty]$,
\[
  \|\imath\|_{F_{s,p,\bsgamma}\hookrightarrow H_{s,p,\bsgamma}}
  \,=\,\Omega\left(s^\tau\right)
\quad\mbox{for all\ }\tau\,>\,0.
\]
\end{proposition}
\begin{proof}
As already mentioned, the result is known for $p\in\{1,\infty\}$ and,
therefore we consider now only $p\in(1,\infty)$.
We use the lower bound from Proposition \ref{prop:ll} for
$\setu=[s]$ with the summation restricted to $\setv=\{k,k+1,\dots,s\}$
for $k=1,2,\dots,s$. Then
\begin{eqnarray*}
  \|\imath\|_{F_{s,p,\bsgamma} \hookrightarrow H_{s,p,\bsgamma}}^p &\ge&
  \sum_{k=1}^s\frac{\gamma_{[s]}^p}{(p^*+1)^{p(k-1)/p^*}\gamma_{\{k,\dots,s\}}^p}\\
  &=&\sum_{k=1}^s\left(\frac{s!}{(s-k+1)!}\right)^{p\,\beta_1}\,
     \left(\frac{c}{(p^*+1)^{1/p^*}}\right)^{p(k-1)}\,
  \frac1{((k-1)!)^{p\ \beta_2}}\\
  &\ge& \sum_{k=1}^s (s-k+2)^{p\,(k-1)\,\beta_1}\,
     \left(\frac{c}{(p^*+1)^{1/p^*}}\right)^{p(k-1)}\,
  \frac1{((k-1)!)^{p\ \beta_2}}.
\end{eqnarray*}
For given $\tau$ consider only one term from the
sum above with $k$ such that
$k-1=\lceil \tau/\beta_1\rceil$. Then for $s\ge k+2$ we have
\[
   \|\imath\|_{F_{s,p,\bsgamma} \hookrightarrow H_{s,p,\bsgamma}}
  \,>\,a_\tau\,\left(s+1-\lceil\tau/\beta_1\rceil\right)^\tau
\]
for $a_\tau=(c/(p^*+1)^{1/p^*})^{\lceil \tau/\beta_1\rceil}
/(\lceil \tau/\beta_1\rceil !)^{\beta_2}.$ This completes the proof.
\end{proof}


\begin{small}
\noindent\textbf{Authors' addresses:}\\

\medskip

\noindent Peter Kritzer\\
Johann Radon Institute for Computational and Applied Mathematics (RICAM)\\
Austrian Academy of Sciences\\
Altenbergerstr.~69, 4040 Linz, Austria\\
E-mail: \texttt{peter.kritzer@oeaw.ac.at}

\medskip

\noindent Friedrich Pillichshammer\\
Institut f\"{u}r Finanzmathematik und Angewandte Zahlentheorie\\
Johannes Kepler Universit\"{a}t Linz\\
Altenbergerstr.~69, 4040 Linz, Austria\\
E-mail: \texttt{friedrich.pillichshammer@jku.at}

\medskip

\noindent G. W. Wasilkowski\\
Computer Science Department, University of Kentucky\\
301 David Marksbury Building\\
329 Rose Street\\
Lexington, KY 40506, USA\\
E-mail: \texttt{greg@cs.uky.edu}
\end{small}

\begin{thebibliography}{99}

\bibitem{C07}
J. Creutzig: Finite-diameter weights. Unpublished, 2007.

\bibitem{DSWW06} J. Dick, I.H. Sloan, X. Wang, and H. Wo\'zniakowski:
Good lattice rules in weighted Korobov spaces with general weights, 
Numer. Math. 103: 63--97, 2006.

\bibitem{HeRi13}
M. Hefter and K. Ritter: On embeddings of weighted tensor product
Hilbert spaces, J. Complexity 31: 405--423, 2015.

\bibitem{HeRiWa15} M. Hefter, K. Ritter, and G. W. Wasilkowski:
On equivalence of weighted anchored and ANOVA spaces
of functions with mixed smoothness of order one in
$L_1$ and $L_\infty$ norms, J. Complexity, electronic version.
DOI: 10.1016/j.jco.2015.07.001

\bibitem{SH}
A. Hinrichs and J. Schneider: Equivalence of anchored and ANOVA
spaces via interpolation, J. Complexity, to appear.

\bibitem{KrPiWa15}
P. Kritzer, F. Pillichshammer, and G. W. Wasilkowski:
Very low truncation dimension for high dimensional integration
under modest error demand, submitted.

\bibitem{KSS12} F.Y. Kuo, C. Schwab, and I.H. Sloan: Quasi-Monte Carlo 
finite element methods for a class of elliptic partial differential equations with 
random coefficients. SIAM J. Numer. Anal. 6, 3351--3374, 2012.

\bibitem{KSWW10b}  F.Y. Kuo, I.H. Sloan, G. W. Wasilkowski, and
H. Wo\'zniakowski: Liberating the dimension, J. Complexity 26:
422--454, 2010.

\bibitem{NW08} E. Novak and H. Wo\'{z}niakowski:
{\it Tractability of multivariate Problems. Volume I: Linear
Information.} European Mathematical Society, Z\"urich, 2008.

\bibitem{Owen14} A. Owen: Effective dimension for weighted function spaces,
  Technical Report, 2014. Available under:\\
  \texttt{http://statweb.stanford.edu/~owen/reports/effdim-periodic.pdf}

\bibitem{SlWo98}
I. H. Sloan and H. Wo\'zniakowski: When are quasi-Monte Carlo
algorithms efficient for high dimensional integrals? 
J. Complexity 14: 1--33, 1998.

\bibitem{Was14}
G. W. Wasilkowski: Tractability of approximation of $\infty$-variate
functions with bounded mixed partial derivatives,
J. Complexity 30: 325--346, 2014.
\end{thebibliography}
\end{document}